\documentclass[12pt]{amsart}
\usepackage{amsfonts,amsmath,amsthm,amssymb}
\usepackage{url}

\newcommand{\nexteq}{\displaybreak[0]\\ &=}

\newtheorem{thm}{Theorem}
\newtheorem{lem}[thm]{Lemma}
\newtheorem{prop}[thm]{Proposition}
\theoremstyle{definition}

\theoremstyle{remark}
\newtheorem{rem}[thm]{Remark}

\DeclareMathOperator{\GF}{GF}
\DeclareMathOperator{\wt}{wt}

\newcommand{\FF}{\mathbb{F}}
\newcommand{\RR}{\mathbb{R}}
\newcommand{\Z}{\mathbb{Z}}
\newcommand{\ZZ}{\mathbb{Z}}
\newcommand{\allone}{\mathbf{1}}
\newcommand{\ep}{\varepsilon}

\begin{document}

\title{On the classification of self-dual $[20,10,9]$ codes over $\GF(7)$}

\author{Masaaki Harada}
\address{Research Center for Pure and Applied
Mathematics, 
Graduate School of Information Sciences, 
Tohoku University, Sendai 980--8579, Japan}
\email{mharada@m.tohoku.ac.jp}
\author{Akihiro Munemasa}
\address{Research Center for Pure and Applied Mathematics, 
Graduate School of Information Sciences, 
Tohoku University, Sendai 980--8579, Japan}
\email{munemasa@math.is.tohoku.ac.jp}
\date{}

\maketitle

\begin{center}
{In memory of Yutaka Hiramine}
\end{center}

\begin{abstract}
It is shown that the extended quadratic residue code of length $20$
over $\GF(7)$
is a unique self-dual $[20,10,9]$ code $C$ such that
the lattice obtained from $C$ by Construction A is isomorphic to
the $20$-dimensional unimodular lattice $D_{20}^+$,
up to equivalence.
This is done by converting the classification of
such self-dual codes to that of
skew-Hadamard matrices of order $20$.
\end{abstract}

\section{Introduction}

Let $\GF(p)$ be the finite field of order $p$, where $p$ is prime.
As described in~\cite{RS-Handbook},
self-dual codes are an important class of linear codes for both
theoretical and practical reasons.
For $p \equiv 1 \pmod 4$,
a self-dual code of length $n$ over $\GF(p)$ exists
if and only if $n$ is even, and
for $p \equiv 3 \pmod 4$,
a self-dual code of length $n$ over $\GF(p)$ exists
if and only if $n \equiv 0 \pmod4$.
It is a fundamental problem to classify self-dual codes 
over $\GF(p)$ and determine the largest minimum weight 
among self-dual codes over $\GF(p)$ for a fixed length.
Much work has been done towards classifying self-dual codes 
over $\GF(p)$ and determining the largest minimum 
weight among self-dual codes of a given length over $\GF(p)$ 
for $p=2$ and $3$ (see~\cite{RS-Handbook}).

Self-dual codes over $\GF(7)$ have been classified 
for lengths up to $12$ (see~\cite{HO02}), 
and the largest minimum weight $d_7(n)$
among self-dual codes of length $n$ over $\GF(7)$
has been determined for $n \le 28$ (see~\cite[Table~2]{GHM}). 
For example, it is known that $d_7(20)=9$ and
the extended quadratic residue code $QR_{20}$ of length $20$ 
over $\GF(7)$ is a self-dual $[20,10,9]$ code (see~\cite{Grassl}).
%
%

There are $12$ nonisomorphic
$20$-dimensional unimodular lattices having minimum norm $2$
(see~\cite[Table~16.7]{SPLAG}),
and one of them is $D^+_{20}$.
Let $A_{7}(C)$ denote the unimodular lattice obtained from
a self-dual code $C$ over $\GF(7)$
by Construction A\@. 

In this paper, we convert the classification of
self-dual $[20,10,9]$ codes $C$ over $\GF(7)$ such that
$A_7(C)$ is isomorphic to $D_{20}^+$ to that of
skew-Hadamard matrices of order $20$.
  The main aim of this paper is to give the following
partial classification of
self-dual $[20,10,9]$ codes over $\GF(7)$.

\begin{thm}\label{thm}
Up to equivalence, 
the extended quadratic residue code of length $20$ over $\GF(7)$
is a unique self-dual $[20,10,9]$ code $C$ over $\GF(7)$ such that
$A_7(C)$ is isomorphic to $D_{20}^+$. 
\end{thm}

All computer calculations in this paper
were done with the help of {\sc Magma}~\cite{Magma}.

\section{Preliminaries}\label{sec:Pre}

In this section, we give definitions and notions on
self-dual codes, unimodular lattices and skew-Hadamard matrices.
Some basic facts on these subjects
are also provided.

\subsection{Self-dual codes}
An {\em $[n,k]$ code} $C$ over $\GF(p)$ is a $k$-dimensional subspace of 
$\GF(p)^n$.
The value $n$ is called the {\em length} of $C$.
The {\em weight} $\wt(x)$ of a vector $x \in \GF(p)^n$ is 
the number of non-zero components of $x$.
A vector of $C$ is called a {\em codeword} of $C$.
The minimum non-zero weight of all codewords in $C$ is called 
the {\em minimum weight} of $C$ and an $[n,k]$ code with minimum 
weight $d$ is called an $[n,k,d]$ code.
The {\em weight enumerator} $W(C)$ of $C$ is given by 
$W(C)= \sum_{i=0}^{n} A_i y^i$, where $A_i$ is the number of codewords of 
weight $i$ in $C$.  
The {\em dual code} $C^{\perp}$ of $C$ is defined as
\[
C^{\perp}=
\{x \in \GF(p)^n \mid x \cdot y = 0 \text{ for all } y \in C\},
\]
under the standard inner product $x \cdot y$.
A code $C$ is  called {\em self-dual} if $C = C^{\perp}$. 
Two codes $C$ and $C'$ are {\em equivalent} if there exists a $(1,-1,0)$-monomial
matrix $M$ with $C' =\{c M \mid c \in C \}$.

\subsection{Unimodular lattices}
An $n$-dimensional (Euclidean) lattice is a discrete subgroup
of rank $n$ in $\RR^n$.
A lattice $L$ is {\em unimodular} if
$L = L^{*}$, where
the dual lattice $L^{*}$ is defined as 
\[
L^{*} = \{ x \in {\RR}^n \mid (x,y) \in \ZZ \text{ for all }
y \in L\},
\]
under the standard inner product $(x,y)$.
The {\em norm} $\|x\|^2$ of a vector $x \in \RR^n$ is $(x, x)$.
The {\em minimum norm} of $L$ is the smallest 
norm among all nonzero vectors of $L$.
Two lattices $L$ and $L'$ are {\em isomorphic}, denoted $L \cong L'$,
if there exists an orthogonal matrix $A$ with
$L' = \{x A \mid x \in L \}$.

Let $C$ be a code of length $n$ over $\GF(p)$ and
let $\ep_1, \ldots, \ep_{n}$ be an orthogonal
basis of $\RR^n$ satisfying $(\ep_i, \ep_j) = p \delta_{i,j}$,
where $\delta_{i,j}$ is the Kronecker delta.
Then we define the lattice $A_{p}(C)$ obtained from $C$ by
{\em Construction A} as
\[
A_p(C)=\{\frac1p\sum_{i=1}^{n}x_i\ep_i\mid
x=(x_1,\dots,x_n)\in\Z^{n},\;x\bmod{p}\in C\}.
\]
It is known that $A_{p}(C)$ is unimodular if and only if
$C$ is self-dual.
A set $\{f_1, \ldots, f_{n}\}$ of $n$ vectors $f_1, \ldots, f_{n}$ in an
$n$-dimensional lattice $L$ with
$(f_i, f_j) = k \delta_{i,j}$
is called a {\em $k$-frame} of $L$.
Clearly, $A_p(C)$ contains a $p$-frame.
Conversely, if a unimodular lattice $L$ contains a $p$-frame, then there is a
self-dual code $C$ over $\GF(p)$ with $A_p(C) \cong L$
(see~\cite{HMV}).

Let $C$ be a self-dual $[20,10,d]$ code over $\GF(7)$ with
$d \in \{8,9\}$.  Then it is easy to see that
$A_7(C)$ has minimum norm $2$.
It is known that there are $12$ nonisomorphic
$20$-dimensional unimodular lattices having minimum norm $2$
(see~\cite[Table~16.7]{SPLAG}), and one of them is $D^+_{20}$.
The lattice $D_{20}^+$
is defined from the root lattice $D_{20}$ as follows:
\begin{align*}
D_{20}&=\{\sum_{i=1}^{20}\alpha_i e_i\mid (\alpha_1,\dots,\alpha_{20})
\in\Z^{20},\;\sum_{i=1}^{20}\alpha_i\equiv0\pmod2\},\\
D_{20}^+&=\langle D_{20},\frac12\allone\rangle,
\end{align*}
where
$e_i=(\delta_{1,i},\ldots,\delta_{20,i})$
$(1 \le i \le 20)$
and $\allone$ denotes the all-one vector.
Note that $D_{20}$ is the even sublattice of $D_{20}^+$, that is,
the sublattice consisting of vectors of even norm in $D_{20}^+$.

\subsection{Skew-Hadamard matrices}
A {\em Hadamard matrix} of order $n$ is an $n \times n$ $(1,-1)$-matrix
$H$ such that $H H^\top = nI$,  
where $I$ is the identity matrix
and $H^\top$ denotes the transposed matrix of $H$.
It is well known that the order $n$ is necessarily $1,2$, or a multiple of $4$.
Two Hadamard matrices $H$ and $K$ are said to be {\em equivalent}
if there are $(1,-1,0)$-monomial matrices $P$ and $Q$ with $K = PHQ$.
All Hadamard matrices of orders up to $32$ have been classified
(see~\cite[Chap.~7]{HSS-OA} for orders up to $28$ and 
\cite{KT13} for order $32$, see also~\cite{Hadamard}).

A Hadamard matrix $H$ of order $n$ is called a {\em skew-Hadamard matrix} 
if $H+H^\top=2I$.
Skew-Hadamard matrices are a class of Hadamard matrices, which
has been widely studied (see e.g., \cite{GS}, \cite{LW91}).
The numbers of inequivalent skew-Hadamard matrices
of orders $4,8,12,16,20,24$ are 
$1, 1, 1, 3, 2, 11$, respectively~\cite{LW91}.
We denote by $S_1$ the Paley Hadamard matrix of order $20$,
which is a skew-Hadamard matrix.
The other skew-Hadamard matrix
of order $20$ can be found in~\cite{skewH}
and we denote the matrix by $S_2$.
Moreover, we have verified with the help of
{\sc Magma} that $S_2$ is equivalent to
\verb+had.20.toncheviv+ in~\cite{Hadamard}.

The following lemma can be proved in the same manner
as~\cite[Lemma~3]{MT}.

\begin{lem}\label{lem:MT}
Let $F$ be a square matrix all of whose entries are integers.
If $FF^\top=kI$ and $p$ is a prime divisor of $k$ such that
$p^2\nmid k$, then $F$ generates a self-dual code
over $\GF(p)$.
\end{lem}

Hence, the code over $\GF(7)$ generated 
by the row vectors of $H+2I$ is self-dual,
where $H$ is a skew-Hadamard matrix of order $20$.

\section{Proof of Theorem~\ref{thm}}

In this section, we give a proof of Theorem~\ref{thm},
which is the main result of this paper.


\begin{lem}\label{lem:1}
Let $C$ be a self-dual $[20,10,9]$ code over $\GF(7)$.
If $\xi\in A_7(C)$ and $\|\xi\|^2=2$, then
\[
|\{i\mid 1\leq i\leq 20,\;|(\xi,\ep_i)|\geq2\}|\leq1.
\]
\end{lem}
\begin{proof}
Write
\[
\xi=\frac{1}{7}\sum_{i=1}^{20}x_i\ep_i,\;
x=(x_1,\dots,x_n)\in\Z^{20},\;x\bmod{7}\in C.
\]
Since for each $j\in\{1,\dots,20\}$,
\begin{align*}
x_j^2&=7\|\frac17 x_j\ep_j\|^2
\\&\leq7\|\frac17 \sum_{i=1}^{20}x_i\ep_i\|^2
\nexteq7\|\xi\|^2
\nexteq14,
\end{align*}
we have 
\begin{equation}\label{eq:lem3}
x_j\equiv0\pmod7\iff x_j=0\iff (\xi,\ep_j)=0. 
\end{equation}
Set
\begin{align*}
a_1&=|\{i\mid 1\leq i\leq 20,\;|(\xi,\ep_i)|=1\}|,\\
a_2&=|\{i\mid 1\leq i\leq 20,\;|(\xi,\ep_i)|\geq2\}|.
\end{align*}
Then by \eqref{eq:lem3} we have
\begin{align*}
a_1+a_2&=\wt(x)\geq9,
\end{align*}
and we have
\begin{align*}
a_1+4a_2&\leq
\sum_{i=1}^{20}(\xi,\ep_i)^2
\nexteq
7\|\xi\|^2
\nexteq
14.
\end{align*}
Thus $a_2\leq\frac53$, and hence $a_2\leq1$.
\end{proof}

\begin{prop}\label{prop:1}
Let $C$ be a self-dual $[20,10,9]$ code over $\GF(7)$ with
$A_7(C)\cong D_{20}^+$. Then there exists a skew-Hadamard matrix
$H$ of order $20$
such that $C$ is generated by the row vectors of $H+2I$ over $\GF(7)$.
\end{prop}
\begin{proof}
Let $\Psi:A_7(C)\to D_{20}^+$ be an isomorphism.   
Since $\|\Psi(\ep_j)\|^2=\|\ep_j\|^2=7$ is odd,
$\Psi(\ep_j)\notin D_{20}$. Thus
$\Psi(\ep_j)\in\frac12\allone+D_{20}\subset\frac12(1+2\Z)^{20}$,
and hence there exist odd integers $f_{i,j}$ such that
\[
\Psi(\ep_j)=\frac12\sum_{i=1}^{20}f_{i,j}e_i.
\]
Let $F$ denote the $20\times20$ matrix whose $(i,j)$ entry is $f_{i,j}$.
Then
$F^\top F=28I$.
In particular,
\[
\sum_{h=1}^{20} f_{h,i}^2=28.
\]
Since $f_{h,i}$ are odd integers, we see that there exists a unique
$h_i$ such that $f_{h_i,i}=\pm3$. 
Since $FF^\top=28I$,
the mapping $i\mapsto h_i$
is a bijection from $\{1,\dots,20\}$ to itself. 
%

Now we may assume without loss of generality
\[
f_{h,i}=\begin{cases}
3&\text{if $h=i$,}\\
\pm1&\text{otherwise}.
\end{cases}
\]
Set $H=F-2I$. Then all the entries of $H$ are $\pm1$,
and the diagonal entries are $1$. 

We claim $H+H^\top=2I$. To prove this, we need to show
$f_{h,i}+f_{i,h}=0$ for $1\leq h<i\leq 20$. Suppose 
$f_{h,i}=f_{i,h}$ for some $1\leq h<i\leq 20$. 
Set $\xi=\Psi^{-1}(e_h+f_{i,h}e_i)$.
Then $\|\xi\|^2=\|e_h+e_i\|^2=2$, and
\begin{align*}
(\xi,\ep_i)&=(\Psi(\xi),\Psi(\ep_i))
\nexteq
(e_h+f_{i,h}e_i,\frac12\sum_{j=1}^{20}f_{j,i}e_j)
\nexteq
\frac12(f_{h,i}+f_{i,h}f_{i,i})
\nexteq
2f_{i,h}.
\end{align*}
Similarly, 
\begin{align*}
(\xi,\ep_h)&=(\Psi(\xi),\Psi(\ep_h))
\nexteq
(e_h+f_{i,h}e_i,\frac12\sum_{j=1}^{20}f_{j,h}e_j)
\nexteq
\frac12(f_{h,h}+f_{i,h}^2)
\nexteq
2.
\end{align*}
These contradict Lemma~\ref{lem:1}, and complete the proof
of the claim.

Since
\begin{align*}
H^\top H&=(F^\top-2I)(F-2I)
\nexteq
28I-2(H^\top+H+4I)+4I
\nexteq
20I,
\end{align*}
$H$ is a Hadamard matrix.

Finally, since
\begin{align*}
D_{20}^+&\ni 2e_i
\nexteq
\frac{1}{14}\sum_{h=1}^{20}28\delta_{h,i}e_h
\nexteq
\frac{1}{14}\sum_{h=1}^{20}\sum_{j=1}^{20}f_{i,j}f_{h,j}e_h
\nexteq
\frac{1}{7}\sum_{j=1}^{20}f_{i,j}\frac12\sum_{h=1}^{20}f_{h,j}e_h
\nexteq
\frac{1}{7}\sum_{j=1}^{20}f_{i,j}\Psi(\ep_j)
\nexteq
\Psi(\frac{1}{7}\sum_{j=1}^{20}f_{i,j}\ep_j),
\end{align*}
we have
\[
\frac{1}{7}\sum_{j=1}^{20}f_{i,j}\ep_j\in A_7(C).
\]
Thus the $i$-th row of $F=H+2I$ belongs to $C$.
The fact that $F$ generates the self-dual code $C$ 
follows from Lemma~\ref{lem:MT}.
\end{proof}

We say that skew-Hadamard matrices $H$ and $H'$ of order $n$
are {\em skew-Hadamard equivalent} if
there exists a $(1,-1,0)$-monomial matrix $P$ with $PHP^\top = H'$.
Let $H$ and $H'$ be skew-Hadamard matrices of order $20$.
Let $C(H)$ denote  the code over $\GF(7)$ generated by
the row vectors of $H+2I$.
If $H$ and $H'$ are skew-Hadamard equivalent, then
$C(H)$ and $C(H')$ are equivalent. 
By Proposition~\ref{prop:1},
we can convert the classification of
self-dual $[20,10,9]$ codes $C$ over $\GF(7)$ with
$A_7(C)\cong D_{20}^+$ to that of
skew-Hadamard matrices of order $20$, up to skew-Hadamard equivalence.
The existence of a skew-Hadamard matrix of order $n$ is equivalent
to the existence of a doubly regular tournament of order $n-1$~\cite{RB72}.
It is known that there are two 
doubly regular tournaments of order $19$, up to isomorphism
(see~\cite{McKay}).
This implies that there are two 
skew-Hadamard matrices of order $20$, up to 
skew-Hadamard equivalence.
Indeed, let $H$ be a skew-Hadamard matrix of order $20$
and let $D$ be the diagonal matrix whose diagonal entries
are the first row of $H$.
Then 
\[
DHD = 
\left( \begin{array}{cc}
1 & \allone \\
-\allone^\top & M
\end{array}
\right).
\]
Here the $19 \times 19$ $(1,0)$-matrix 
$(M+J)/2-I$ is the adjacency matrix of a
doubly regular tournament of order $19$, where
$J$ is the $19 \times 19$ all-one matrix.
Hence, isomorphic doubly regular tournaments of order $19$
give skew-Hadamard matrices of
order $20$, which are skew-Hadamard equivalent.
The matrices $S_1$ and $S_2$ give
the two skew-Hadamard matrices of order $20$, up to
skew-Hadamard equivalence.

We have verified with the help of {\sc Magma} that the two self-dual codes
$C(S_1)$ and $C(S_2)$ have the following weight enumerators:
\begin{align*}
W(C(S_1))=
&
 1
+ 6840 y^{9}
+ 47880 y^{10}
+ 200640 y^{11}
+ 957600 y^{12}
\\ &
+ 3625200 y^{13}
+ 10766160 y^{14}
+ 25701984 y^{15}
\\ &
+ 48495600 y^{16}
+ 68276880 y^{17}
+ 68299680 y^{18}
\\ &
+ 43155840 y^{19}
+ 12940944 y^{20},
\\
W(C(S_2))=
&
1
+ 1080 y^{8}
+ 5040 y^{9}
+ 40320 y^{10}
+ 215760 y^{11}
\\ &
+ 977040 y^{12}
+ 3571200 y^{13}
+ 10751040 y^{14}
\\ &
+ 25814304 y^{15}
+ 48431880 y^{16}
+ 68208840 y^{17}
\\ &
+ 68403000 y^{18}
+ 43106160 y^{19}
+ 12949584 y^{20},
\end{align*}
respectively.
In particular, $C(S_1)$ is a $[20,10,9]$ code, while
$C(S_2)$ has minimum weight $8$. 
By Proposition~\ref{prop:1}, $C(S_1)$ is a unique
self-dual $[20,10,9]$ code $C$ over $\GF(7)$ with
$A_7(C) \cong D_{20}^+$. 
In addition, 
we have verified with the help of
{\sc Magma} that 
$A_7(QR_{20}) \cong D_{20}^+$.
This completes the proof of Theorem~\ref{thm}.

\begin{rem}
The above argument yields that $C(S_1)$ is equivalent to $QR_{20}$.   
A general case including this fact was described 
in~\cite[p.~1041]{Chapman} without proof.
\end{rem}

\section{Some other constructions of self-dual $[20,10,9]$ codes}

Finally, in this section, we investigate some other constructions of
self-dual $[20,10,9]$ codes over $\GF(7)$.
%
Note that $D^+_{20}$ is the unique
$20$-dimensional unimodular lattice with minimum norm $2$ and 
kissing number $760$.
For a given self-dual $[20,10,9]$ code $C$ over $\GF(7)$, 
one can determine whether $A_7(C)$ is isomorphic to $D^+_{20}$ or not,
by computing the kissing number of $A_7(C)$ with the help of {\sc Magma}.
If $A_7(C)$ is isomorphic to $D^+_{20}$, then by Theorem~\ref{thm},
we have that $C$ is equivalent to $QR_{20}$. 

\begin{itemize}
\item 
Some self-dual $[20,10,9]$ codes over $\GF(7)$ were constructed
in~\cite[Table~6]{GH99} and~\cite[Table~7]{GHM}
as double circulant codes and quasi-twisted codes,
respectively (see~\cite{GHM} for the construction).
We have verified that $A_7(C) \cong D^+_{20}$
for all double circulant self-dual $[20,10,9]$ codes $C$.
Also, we have verified that $A_7(C) \cong D^+_{20}$
for all quasi-twisted self-dual $[20,10,9]$ codes $C$.
These imply that all double circulant self-dual $[20,10,9]$
codes and all quasi-twisted self-dual $[20,10,9]$ codes are
equivalent to $QR_{20}$.


\item
Let $A$ and $B$ be $5 \times 5$ circulant (resp.\ negacirculant)
matrices.
A $[20,10]$ code over $\GF(7)$ with the following generator matrix
\[
\left(
\begin{array}{cccc}
\quad & {\Large I} & \quad &
\begin{array}{cc}
 A & B \\
-B^\top & A^\top
\end{array}
\end{array}
\right)
\]
is called a four-circulant (resp.\ four-negacirculant) code.
By exhaustive search, 
we have verified that $A_7(C) \cong D^+_{20}$
for all four-circulant self-dual $[20,10,9]$ codes $C$.
Also, we have verified that $A_7(C) \cong D^+_{20}$
for all four-negacirculant self-dual $[20,10,9]$ codes $C$.

\item
Let $C$ be a self-dual code of length $20$ over $\GF(7)$.
Let $x$ be a vector with $x \cdot x =0$.
Then $C(x)=\langle C \cap \langle x \rangle^\perp, x \rangle$
is a self-dual code over $\GF(7)$.
By exhaustive search, 
we have verified that $A_7({QR_{20}}(x)) \cong D^+_{20}$
for all vectors $x$ in a set of complete representatives of
$\GF(7)^{20}/QR_{20}$ with $x\cdot x=0$.

\end{itemize}

Moreover,
our extensive search failed to discover a self-dual 
$[20,10,9]$ code $C$ over $\GF(7)$ with $A_7(C) \not\cong D^+_{20}$.
We are lead to conjecture that $QR_{20}$ is a unique
self-dual $[20,10,9]$ code over $\GF(7)$.

\bigskip
\noindent 
{\bf Acknowledgments.}
The authors would like to thank Sho Suda for useful discussions.
The authors would also like to thank the anonymous referee
for helpful comments.
This work is supported by JSPS KAKENHI Grant Number 26610032.



\begin{thebibliography}{99}

\bibitem{Magma}W. Bosma, J. Cannon and C. Playoust, 
The Magma algebra system I: The user language, 
{\sl J. Symbolic Comput.}
{\bf 24} (1997), 235--265.

\bibitem{Chapman} R. Chapman,
Conference matrices and unimodular lattices,
{\sl European J. Combin.}
{\bf 22}  (2001), 1033--1045.

\bibitem{SPLAG} J.H. Conway and N.J.A. Sloane, 
{\sl Sphere Packing, Lattices and Groups (3rd ed.)},
Springer-Verlag, New York, 1999.

\bibitem{GS}A.V. Geramita and J. Seberry, 
{\sl Orthogonal Designs, Quadratic Forms and Hadamard Matrices}, 
Marcel Dekker Inc., New York, 1979.

\bibitem{Grassl} M. Grassl, 
{\sl Bounds on the Minimum Distance of Linear Codes and Quantum Codes},
published electronically at
\url{http://www.codetables.de/}.

\bibitem{GH99} T.A. Gulliver and M. Harada,
New optimal self-dual codes over GF(7),
{\sl Graphs Combin.}
{\bf 15} (1999), 175--186.

\bibitem{GHM} T.A. Gulliver, M. Harada and H. Miyabayashi, 
Double circulant and quasi-twisted self-dual codes over $\FF_5$ and $\FF_7$,
{\sl Adv.\ Math.\ Commun.}
{\bf 1} (2007),  223--238.

\bibitem{HMV} M. Harada, A. Munemasa and B. Venkov, 
Classification of ternary extremal self-dual codes of length $28$,
{\sl Math.\ Comp.}
{\bf 78} (2009), 1787--1796.

\bibitem{HO02} M. Harada and P.R.J. \"Osterg\aa rd,
Self-dual and maximal self-orthogonal codes over $\FF_7$,
{\sl Discrete Math.}
{\bf 256} (2002), 471--477.

\bibitem{HSS-OA} A.S. Hedayat, N.J.A. Sloane and J. Stufken, 
{\sl Orthogonal Arrays},
Springer-Verlag, New York, 1999.

\bibitem{KT13} H. Kharaghani and B. Tayfeh-Rezaie, 
Hadamard matrices of order 32,
{\sl J. Combin.\ Des.}
{\bf 21} (2013), 212--221.


\bibitem{skewH} C. Koukouvinos,
{\sl Yamada-Williamson Constructions of nXn Skew-Hadamard Matrices 
for n=4,12,20,28,36,44,52,60,68,76,84,92,100},
published electronically at
\url{http://rangevoting.org/SkewHad.html}.

\bibitem{LW91} C. Lin and W.D. Wallis, 
Symmetric and skew equivalence of Hadamard matrices,
{\sl Congr.\ Numer.}
{\bf 85} (1991), 73--79.

\bibitem{McKay}B. McKay,
{\sl Digraphs},
published electronically at
\url{http://cs.anu.edu.au/~bdm/data/digraphs.html}.

\bibitem{MT} A. Munemasa and H. Tamura,
The codes and the lattices of Hadamard matrices,
{\sl European J. Combin.}
{\bf 33} (2012), 519--533.


\bibitem{RS-Handbook} E. Rains and N.J.A. Sloane,
{``Self-dual codes,''} {Handbook of Coding Theory},
V.S. Pless and W.C. Huffman (Editors),
Elsevier, Amsterdam 1998, pp.\ 177--294.

\bibitem{RB72} K.B. Reid and E. Brown, 
Doubly regular tournaments are equivalent to skew Hadamard matrices,
{\sl J. Combin.\ Theory Ser.~A}
{\bf 12} (1972), 332--338.

\bibitem{Hadamard} N.J.A. Sloane,
{\sl A Library of Hadamard Matrices},
published electronically at
\url{http://neilsloane.com/hadamard/index.html}.

\end{thebibliography}
\end{document}